\documentclass[12pt,reqno]{amsart}
\usepackage{graphicx,hyperref}
\usepackage{amsthm,amsfonts,latexsym,amssymb,cite}
\newtheorem{lemma}{\bf{Lemma} }[section]
\newtheorem{proposition}{\bf{Proposition}}[section]
\newtheorem{theorem}{\bf{Theorem}}[section]
\newtheorem{remark}{\sc{Remark} }[section]

\def\build#1_#2{\mathrel{\mathop{\kern 0pt#1}\limits_{#2}}}
  
\begin{document}

\title{Explicit estimates on a mixed Neumann-Robin-Cauchy problem}
\author{Luisa Consiglieri}
\address{Luisa Consiglieri, Independent Researcher Professor, European Union}
\urladdr{\href{http://sites.google.com/site/luisaconsiglieri}{http://sites.google.com/site/luisaconsiglieri}}

\begin{abstract}
We  deal with the existence of weak solutions
for a mixed Neumann-Robin-Cauchy problem.
The existence results
are based on global-in-time  estimates of approximating solutions, and
 the passage to the limit exploits
compactness techniques.
We investigate explicit estimates for solutions of the
parabolic equations with nonhomogeneous boundary conditions and
distributional right hand sides.
The parabolic equation is of divergence form with discontinuous coefficients.
We consider a nonlinear condition on a part of the boundary that the power laws 
(and the Robin boundary condition) appear as particular cases.
\end{abstract}


\maketitle

\section{Introduction}

The existence of solutions to partial differential equations (PDE)
is not sufficient whenever the main objective is their application to other branches of the science.
In industrial applications, the physical fields (such as temperature or
potentials) verify PDE in divergence form with a nondifferentiable
leading coefficient. Thus, they do not correspond to the classical
solutions. There is a growing demand for 
the existence of quantitative estimates with explicit constants
due to the application of fixed point arguments \cite{freh,homb,safa}.
The knowledge of the values of the involved constants in the estimates is crucial.

The study of the Dirichlet-Cauchy problem is vast in the literature,
\cite{amann,bv79,giaq-str,lapt,lieb,liu,marino,tolk} to mention a few.
It is known that 
Dirichlet boundary conditions  roughly approximate the reality.
 As a consequence,  the study of the Cauchy problem under
Neumann or Robin boundary conditions has its actuality in the
works \cite{akagi,arwarma,choi,hofm,hung,nitt11,nitt13,wang}.

This work is devoted to the determination of the involved constants
for  the boundary value problems concerned in
the presence of radiative-type conditions on the
 boundary which are typical of
thermodynamic models evolved from engineering practice \cite{shill,druet}. 

The derivation of the estimates is not unique.
It depends on the mathematical choice of what are the most relevant data.
Of course, the most relevant data do not come from a mathematical choice,
but from a bio-chemico-geo-physical choice. With this state of mind,
we detail the proofs in order to be easily changed for other requisites.

The steady-state study can be found in \cite{arxA,arxB}.
 
 Let $[0, T] \subset {\mathbb R}$ be the time interval with $ T >0$,
and   $\Omega\subset \mathbb{R}^n$ ($n\geq 2$) be a bounded  domain
 of class $C^{0,1}$. 
The boundary $\partial\Omega$ is decomposed into two 
 disjoint open subsets, namely $\Gamma$ and  $\partial\Omega\setminus
\bar \Gamma$.
 Moreover we set $Q_T=\Omega\times ]0,T[$, and $\Sigma_T=\Gamma\times ]0,T[$.
Here, we consider the 
nonlinear boundary condition version of the Cauchy problem studied in
\cite{bocc}
\begin{align}
\partial_t u-\nabla\cdot(   \mathsf{A}\nabla u)=
-\nabla\cdot(u{\bf E})&\mbox{ in }Q_T;\label{omega}\\
(\mathsf{A}\nabla u-u{\bf E})\cdot{\bf n}=0&\mbox{ on }(\partial
\Omega\setminus \bar\Gamma)\times ]0,T[; \\
(\mathsf{A}\nabla u-u{\bf E})\cdot{\bf n}+b(u)u
=h&\mbox{ on }\Gamma\times ]0,T[, \label{gama}
\end{align}
for  matrix and vector value functions $\mathsf{A}$ and $\bf E$, respectively.
Here, 
$\bf n$ is the outward unit normal to the boundary $\partial\Omega$.
For the sake of simplicity, we handle with no right hand side data,
namely $\nabla\cdot{\bf f}+f$. We refer that these data can be clearly
included in our main results, as well as some lower order terms
in the differential operator.

The function $b$  satisfies a $(\ell-2)$-growth condition
that includes
\begin{itemize}
\item Neumann: $b(u)\equiv 0$;
\item Robin: $b(u)=b_*$ which constant stands for whether the heat convective
transfer coefficient either the Rayleigh-Jeans radiation approximation;
\item Blackbody radiation: $b(u)=\sigma |u|^3$, with $\sigma$ representing
the Stefan-Boltzmann constant;
\item Wien displacement law: $b(u)=b_*u^4$.
\end{itemize}

\section{Main results}

Let us introduce the following Banach spaces, for $p,q > 1$,
 in the framework of Bochner,
Sobolev and Lebesgue functional spaces:
\begin{align*}
V_{p,q}&=\{v\in W^{1,p}(\Omega):\ 
v|_\Gamma\in L^{q}(\Gamma)\};\\
L^{p,q}(Q_T)&=L^q(0,T;L^p(\Omega));\\
L^{p,q}(\Sigma_T)&=L^q(0,T;L^p(\Gamma));\\
W^{1,q}(0,T;X,Y)&=\{v\in L^q(0,T;X):\ v'\in L^q(0,T;Y)\},
\end{align*}
where $X$ and $Y$ denote Banach spaces such that $X\hookrightarrow Y$.

\begin{remark}\label{dtc}
If there exists a Hilbert space $H$ such that $X\hookrightarrow H=\bar
X\hookrightarrow Y=X'$ then (see, for instance, \cite[p. 106]{show})
\[
W^{1,2}(0,T;X,Y)\hookrightarrow C([0,T];H).
\]
Moreover, if $u,v\in W^{1,2}(0,T;X,X')$ then $(u(t),v(t))_H$ is
absolutely continuous on $[0,T]$, and
\begin{equation}\label{perp}
{d\over dt}(u(t),v(t))_H=\langle u'(t),v(t)\rangle+
\langle v'(t),u(t)\rangle, \quad\mbox{a.e. }t\in [0,T].
\end{equation}
\end{remark}

Throughout this paper, the hypothesis on the coefficients
$\mathsf{A}$ and $b$ are
\begin{description}
\item[(A)] the  $(n\times n)$ matrix-valued function
 $\mathsf{A}=[A_{ij}]_{i,j=1,\cdots,n}$ is
measurable, uniformly elliptic, and uniformly bounded:
\begin{align}\label{amin}
\exists  a_\#>0,\quad &
A_{ij}(x)\xi_i\xi_j\geq a_\#|\xi|^2,
\quad\mbox{ a.e. }x\in\Omega,\ \forall \xi\in\mathbb{R}^n;\\
\exists  a^\#>0,\quad &\|\mathsf{A}\|_{\infty,\Omega}\leq a^\#,\label{amax}
\end{align}
under the summation convention over repeated indices:
$\mathsf{A}{\bf a}\cdot{\bf b}=A_{ij}a_jb_i={\bf b}^\top  \mathsf{A}{\bf a}$.
\item[(B)]
 $b:\Gamma\times \mathbb{R}\rightarrow \mathbb{R}$ is a Carath\'eodory function,
 i.e. measurable with respect to $x\in\Gamma$ and
  continuous with respect to $\xi\in\mathbb R$,
such that there exists $\ell\geq 2$ that  $b$ has
$(\ell-2)$-growthness property, and 
it is monotone with
respect to the last variable:
\begin{align}\label{bmm}
\exists  b_\#,b^\# \geq 0, \qquad  b_\#|\xi|^{\ell-2}
\leq b(x,\xi)\leq  b^\#|\xi|^{\ell-2};\\
(b(x,\xi)\xi-b(x,\eta)\eta)(\xi-\eta)\geq 0,\label{bmon}
\end{align} 
for a.e. $x\in\Gamma$, and for all $\xi,\eta\in \mathbb{R}$.
\end{description}

Let us state our first existence result (the Dirichlet problem
is established in \cite[Lemma 3.2]{bocc}).
\begin{theorem}\label{teom}
Let ${\bf E}\in L^{2/(1-n/q-\theta)}(0,T;{\bf L}^q(\Omega))$
with  $\theta=0$ if $n>2$ and
any $0<\theta<1-2/q$ if $n=2$, $h\in L^{\ell\,'}(\Sigma_T)$, 
with $\ell\,'$ standing for the conjugate exponent $\ell\,'=\ell/(\ell-1)$,
and $u_0\in
L^2(\Omega)$. Under the assumptions (A)-(B) with $b_\#>0$,
 there exists a function $u$ in $L^{2,\infty}(Q_T)\cap 
L^{2}(0,T;V_{2,\ell})\cap 
L^{\ell}(\Sigma_T)$,
which is solution of (\ref{omega})-(\ref{gama}) in the sense that
\begin{align}
\langle\partial_t u,v\rangle +\int_{Q_T}
 \mathsf{A}\nabla u\cdot \nabla v\mathrm{dx}\mathrm{dt}
+\int_{\Sigma_T}b(u)u  v\mathrm{ds}\mathrm{dt}=\nonumber\\
=\int_{Q_T} u{\bf E}\cdot \nabla v\mathrm{dx}\mathrm{dt}
+\int_{\Sigma_T}h  v\mathrm{ds}\mathrm{dt},\label{pbu}
\end{align}
for every $v\in L^2(0,T;V_{2,\ell})\cap L^{\ell}(\Sigma_T)$. 
Here, $\partial_t u$ belongs to $L^2(0,T;(V_{2,\ell})')+
L^{\ell/(\ell-1)}(\Sigma_T)$.
In particular, we have
\begin{align}\label{cotauinf}
\|u\|_{2,\infty,Q_T}^2 &\leq \left(\|u_0\|_{2,\Omega}^2+
{2\over \ell\,'b_\#^{1/(\ell-1)}}\|h \|_{\ell\,',\Sigma_T} ^{\ell\,'}\right)
\exp\left[\mathcal{Q}\right];
\\ \label{cotauu}
a_\#\|\nabla u\|_{2,Q_T}^2+
b_\#\|u\|_{\ell,\Sigma_T}^\ell &\leq \left(\|u_0\|_{2,\Omega}^2+
{2\over \ell\,'b_\#^{1/(\ell-1)}}\|h \|_{\ell\,',\Sigma_T} ^{\ell\,'}\right)
\left(\mathcal{Q}
\exp\left[\mathcal{Q}\right]+1\right),
\end{align}
with
\begin{align*}
\mathcal{Q}={2\over a_\#}
S_{(2+\theta q)/(1+\theta q)}^{2(n+\theta q)/q}|\Omega|^{\theta}
\|{\bf E}\|_{q,2,Q_T}^2+\\
+[(4/a_\#)^{q+n+\theta q}
S_{(2+\theta q)/(1+\theta q)}^{2(n+\theta q)}
|\Omega |^{q\theta}]^{(q-n-\theta q)^{-1}}
\int_0^T\|{\bf E}\|_{q,\Omega}^{2(1-n/q-\theta )^{-1}}\mathrm{dt}.
\end{align*}
\end{theorem}
\begin{remark}
Observe that Theorem \ref{teom} remains true if  ${\bf E}\in L^r(0,T;{\bf L}^q(\Omega))$
with $2/r+n/q\leq 1$ if $n>2$, and  $2/r+2/q<1$ if $n=2$.
\end{remark}
Hereinafter, $S_p$ denotes the Sobolev constant of continuity under the standard
$W^{1,p}(\Omega)$-norm $(1\leq p<n)$.
Meanwhile, 
$S_{p,q}$ denotes the Sobolev constant of continuity
under the $V_{p,q}$-norm. Other constants occur in $\mathcal{Q}$ if we use
the inequality \cite{maggi}:
\[
\|v\|_{pn/(n-p),\Omega}\leq S_p
\|\nabla v\|_{p,\Omega}+S_1^{1/p_*}
\|v\|_{p_*,\partial\Omega }.
\]
Notice that $V_{p,\ell}=W^{1,p}(\Omega)$
whenever the radiation exponent $\ell \leq p_*=p(n-1)/(n-p)$ if $n>p$.

Next, we establish a maximum principle due to the Moser technique
(see, for instance, \cite{aronson}), with the upper bound being different from
the one established in \cite[Theorem 2.1]{bocc} which
depends on the data in an exponential form,
which is a shortcoming for physical applications.
\begin{theorem}\label{teomm}
 Under   $2/r+n/q<1$,
 $h\in L^\infty(\Sigma_T)$, $h\geq 0$ on $\Sigma_T$, and  $u_0\in
L^\infty(\Omega)$, $u_0>0$ in
$\Omega$, any solution in accordance with Theorem \ref{teom}
satisfies
$ 0\leq u\leq \mathcal{M}$ in $Q_T$, and $ 0\leq u
\leq \left(\mathcal{M}+P_1\right)/b_\#$ 
on $\Sigma_T$, 
if provided by the smallness condition $P_2\leq P$,
with  $\mathcal{M}$, $P_1$, $P_2$, and $P$ being explicitly
given in Proposition \ref{pmax}.
\end{theorem}

We state the following existence result (see for instance
\cite[Section 4.4]{lap} in where the divergence free
$\bf E$ is taken into account). We emphasize that the estimate (\ref{cotauup})
is not so
pleasant as we might expect.
\begin{theorem}\label{teomf}
Let  $u_0\in L^1(\Omega)$, $f\in L^1(Q_T)$,
$h\in L^1(\Sigma_T)$, and ${\bf E}\in L^r(0,T;{\bf L}^q(\Omega))$ for
\begin{equation}\label{rnq}
1<r\left(1-{n\over q}\right)<2.
\end{equation}
Under the assumptions (A)-(B) with $b_\#>0$,
 there exists a function $u$ in $L^{1,\infty}(Q_T)\cap 
L^{p}(0,T;V_{p,\ell-1})\cap L^{\ell-1}(\Sigma_T)$
 such that $\partial_t u\in
L^{1}(0,T;[W^{1,p'}(\Omega)]')$, with 
\begin{equation}\label{pqrn}
{n\over q}+{p(n+1)-n\over r}= 1,
\end{equation} 
 satisfying  the variational problem
\begin{align}
\langle\partial_t u,v\rangle +\int_{Q_T}
 \mathsf{A}\nabla u\cdot \nabla v\mathrm{dx}
\mathrm{dt}
+\int_{\Sigma_T}b(u)u  v\mathrm{ds}\mathrm{dt}=\nonumber\\
=\int_{Q_T} u{\bf E}\cdot \nabla v\mathrm{dx}\mathrm{dt}
+\int_{Q_T}fv\mathrm{dx}\mathrm{dt}
+\int_{\Sigma_T}h  v\mathrm{ds}\mathrm{dt}\label{pbuf}
\end{align}
for every $v\in L^\infty(0,T;W^{1,p'}(\Omega))$. For $r(2-p)<2np$, we have
\begin{align}\label{cotauinfp}
\|u\|_{1,\infty,Q_T}+b_\#
\|u\|_{\ell-1,\Sigma_T}^{\ell-1}
 \leq \|u_0\|_{1,\Omega}+
\|f\|_{1,Q_T}+\|h \|_{1,\Sigma_T} := \mathcal{Z};
\\ \label{cotauup}
\|\nabla u\|_{p,Q_T}^p \leq \mathcal{B}
+{rn^2 \left((Z_1)^{2-p} \mathcal{Z}^{(2-p)/ n}+1\right)
\over a_\# (n+2-p(n+1)) (n-1)} 
\left( {b^\#\over b_\#}\mathcal{Z}\right),
\end{align}
with 
\begin{align*}
\mathcal{B}=r\left( T|\Omega|+(Z_2)^{p}\mathcal{Z}^{p(n+1)/n}
\right)+\\
+{rn^2\over a_\# (n+2-p(n+1))(n-1)} 
\left((Z_1)^{2-p} \mathcal{Z}^{2-p\over n}+1\right)
\left( T|\Omega|+2\mathcal{Z}\right)+\\
+{r\over 2(a_\#)^2}\|{\bf E}\|_{q,r,Q_T}^2
(Z_2)^{p(r-2)/r}\left( (Z_1)^{2-p}\mathcal{Z}^p+\mathcal{Z}^{[p(n+1)-2]/n}
\right)+\\
+{2^{(r-2)/2}\over (a_\#)^r}\|{\bf E}\|_{q,r,Q_T}^r\left(
 (Z_1)^{r-p}\mathcal{Z}^p
+(Z_1)^{p(r-2)/2}\mathcal{Z}^{p-r(2-p)/(2n)}
\right),
\end{align*}
where $Z_1=S _p (1+|\Omega|^{1/n}S_1)$ and
$Z_2= S_pS_1T^{1/p} |\Omega|^{1/p+1/n-1}$.
\end{theorem}
Observe that (\ref{rnq})-(\ref{pqrn}) mean $1<p<(n+2)/(n+1)$.

Under similar  proofs,
we state the corresponding results of Theorems \ref{teom},
and \ref{teomf} under the assumption (\ref{bmm}) with $b_\#=0$.
In the following, 
$K_p$ denotes the constant of continuity of the embedding
$W^{1,p}(\Omega)\hookrightarrow L^{p_*}(\Gamma)$ ($p< n$).
\begin{theorem}\label{tb0}
If the conditions of Theorem \ref{teom} are fulfilled
under $h \in L^{(2_*)',2}(\Sigma_T)$,
where $2_*=2(n-1)/(n-2)$ if $n>2$ and $2_*$ represents any real number 
greater than $2$,
and the assumption (\ref{bmm}) with $b_\#=0$, and $2\leq\ell\leq 3$,
 then the variational problem (\ref{pbu}) admits
at least one solution $u\in L^{2,\infty}(Q_T)\cap 
W^{1,2}(0,T;H^1(\Omega);[H^1(\Omega)]')$ satisfying the following estimates
\begin{align}\label{cotaub0}
\|u\|_{2,\infty,Q_T} &\leq \mathcal{A}
\sqrt{\exp\left[\mathcal{Q}+T\right]};
\\ \label{cotaub1}
a_\#\|\nabla u\|_{2,Q_T}^2 &\leq 2\mathcal{A}^2
\left((\mathcal{Q}+T)
\exp\left[\mathcal{Q}+T\right]+1\right),
\end{align}
with $\mathcal{M}$ according to Theorem \ref{teom},
and
\[\mathcal{A}^2=\|u_0\|_{2,\Omega}^2+
(2/a_\#+1)
(K_s)^2|\Omega |^{2/s-1}\|h \|_{(2_*)',2,\Sigma_T}^2,
\]
where  $s=2$ if $n>2$ and $s=22_*/(2_*+1)$ if $n=2$.
\end{theorem}
\begin{theorem}\label{tb0f}
If the conditions of Theorem \ref{teomf} are fulfilled
under the assumption (\ref{bmm}) with $b_\#=0$,
then the variational problem (\ref{pbuf}) admits
at least one solution $u\in L^{1,\infty}(Q_T)\cap 
L^{p}(0,T;W^{1,p}(\Omega))$,  
 under (\ref{pqrn})  and $2\leq\ell\leq p+1$, such that $\partial_t u\in
L^{1}(0,T;[W^{1,p'}(\Omega)]')$,
 satisfying the following estimates
\begin{align}\label{cotaub0p}
\|u\|_{1,\infty,Q_T}
 \leq \|u_0\|_{1,\Omega}+
\|f\|_{1,Q_T}+\|h \|_{1,\Sigma_T} := \mathcal{Z};
\\ 
\|\nabla u\|_{p,Q_T}^p \leq \alpha_\ell 
+\beta_\ell\left(\mathcal{B}
 +2^{\ell-2}\beta\left(
(S_{1}^{n(p-1)\over n-p(n-1)}|\Omega|^{n(p-1)^2\over (n-p(n-1))p}
+S_{1}^{n(p-1)\over p})T\mathcal{Z}\right)^{\ell-1}\right),\label{cotauub0p}
\end{align}
with 
\begin{align*}
\left\{\begin{array}{l}
\alpha_\ell=0,\ \beta_\ell=(1-2^{2p-1}\beta)^{-1}\quad\mbox{ if }\ell=p+1
\mbox{ and } \beta< 2^{1-2p}\\
\alpha_\ell=(2^{2\ell-3}\beta)^{p/(p-\ell+1)},\ \beta_\ell=
p/(p-\ell+1)\quad\mbox{ if }\ell<p+1\\
\end{array}\right.\\
\beta= b^\#
T^{1-{(\ell-1)/ p}} |\Gamma |^{1-{(\ell-1)/ p_*}}
K_p^{\ell-1}{rn^2 \left((Z_1)^{2-p} \mathcal{Z}^{(2-p)/ n}+1\right)
\over a_\# (n+2-p(n+1)) (n-1)} .
\end{align*}
\end{theorem}
\begin{remark}
The Neumann problem, i.e. $b_\#=b^\#=0$, clearly verifies the smallness condition $\beta=0$. Then, (\ref{cotauub0p}) is satisfied under
$\alpha_\ell=0$, $\beta_\ell=1$, and $\ell=p+1$.
\end{remark}

Finally, we restrict to the minimum principle (cf. Proposition \ref{pmin}).
The explicit upper bound correspondent to $\mathcal{M}$
in Teorem  \ref{teomm} is not straightforward, remaining as open problem.
\begin{theorem}\label{tb0m}
If  $h\in L^{2}(\Sigma_T)$, $h\geq 0$ on $\Sigma_T$, and  $u_0\in
L^2(\Omega)$, $u_0>0$ in $\Omega$, 
any solution in accordance with Theorem \ref{tb0}
is nonnegative in $Q_T$, and its trace is nonnegative on $\Sigma_T$.
\end{theorem}

\section{Proof of Theorem \ref{teom}}
\label{steom}

The proof of existence is divided into three canonical steps: existence of approximate solutions (regularization),
derivation of uniform estimates, and passage to the limit.

For each $m\in\mathbb{N}$, if we consider the truncating function 
\begin{equation}\label{trt}
T_m(s)=\min\{m,\max\{-m,s\}\}\quad\mbox{ for }s\in\mathbb{R},
\end{equation}
then there exists at least a weak solution $u_m$ of
\begin{align}
\int_0^T\langle\partial_t u_m,v\rangle\mathrm{dt} +\int_{Q_T}
( \mathsf{A}\nabla u_m+T_m(u_m){\bf E}) \cdot \nabla v\mathrm{dx}
\mathrm{dt}+\nonumber\\
+\int_{\Sigma_T} b(u_m)u_mv\mathrm{ds}\mathrm{dt}
=\int_{\Sigma_T} h v\mathrm{ds}\mathrm{dt},\qquad\forall v\in L^2(0,T;
V_{2,\ell}),\label{pbum}
\end{align}
which belongs to
$ L^{2,\infty}(Q_T)\cap L^2(0,T;V_{2,\ell})$ such that
 $\partial_t u_m\in L^2(0,T;(V_{2,\ell})')$. The existence
is true due to the Faedo-Galerkin method \cite[Theorem 4.1, p. 120]{show}.

In order to pass to the limit as $m$ tends to infinity,
we seek for  estimates independent on $m$.

\subsection{Proof of the estimates (\ref{cotauinf})-(\ref{cotauu})
for $u_m$}
\label{est}

 Let us take  $v=\chi(t,\tau)u_m$
 as a test function in (\ref{pbum}), where
 $\chi(t,\tau)$ is the characteristic function of the open interval $]0,\tau[$,
with $\tau$ being a fixed number lesser than $T$.
Applying the assumptions
(\ref{amin}) and (\ref{bmm}) with $b_\#>0$, it follows that
\begin{align}
{1\over 2}\int_{\Omega}|u_m|^2(\tau)\mathrm{dx}
+a_\#\int_{Q_\tau}|\nabla u_m|^2\mathrm{dx}
\mathrm{dt}+b_\#\int_{\Sigma_T} |u_m|^\ell\mathrm{ds}\mathrm{dt}\leq\nonumber\\
\leq 
{1\over 2}\|u_0\|_{2,\Omega}^2+
\int_0^{\tau} \|u_m\|_{2q/(q-2),\Omega}\|{\bf E}\|_{q,\Omega}
\|\nabla u_m\|_{2,\Omega}\mathrm{dt}
+\|h \|_{\ell\,',\Sigma_\tau} \|u_m\|_{\ell,\Sigma_\tau} \leq\nonumber\\
\leq 
{1\over 2}\|u_0\|_{2,\Omega}^2+I
+{1\over \ell\,'b_\#^{1/(\ell-1)}}\|h \|_{\ell\,',\Sigma_\tau} ^{\ell\,'}+
{b_\#\over \ell}\|u_m\|_{\ell,\Sigma_\tau} ^\ell,\label{stepm}
\end{align}
by considering  the property $|T_m(u)|\leq |u|$,
and the H\"older and Young inequalities.

For any $0<\lambda<1$ such that
\[
\lambda={2s\over(s-2)q}\quad \left(\Leftrightarrow
{q-2\over 2q}={\lambda \over s}+{1-\lambda\over 2}\right),
\]
the H\"older inequality yields
\begin{equation}\label{hold}
 \|u_m\|_{2q/(q-2),\Omega}\leq  \|u_m\|_{s,\Omega}^\lambda
  \|u_m\|_{2,\Omega}^{1-\lambda}.
  \end{equation}
Taking $\lambda=n/q+\theta$ with $\theta=0$ if $n>2$ and
any $0<\theta<1-2/q$ if $n=2$  to uniform the Sobolev constants for dimensions
$n>2$ and $n=2$, we use 
 the Sobolev embedding $W^{1,ns/(s+n)}(\Omega) \hookrightarrow L^s(\Omega)$
followed by the H\"older inequality 
\begin{equation}\label{sn2}
\|u_m\|_{s,\Omega}\leq S_{ns/(s+n)}|\Omega |^{1/s+1/n-1/2}\left(
\|\nabla u_m\|_{2,\Omega}+\|u_m\|_{2,\Omega}\right),
\end{equation}
where $s=2(n+\theta q)(n-2+\theta q)^{-1}$.
Gathering (\ref{hold}) and (\ref{sn2}) we deduce
\begin{align*}
I	
\leq {a_\#\over 2}\|\nabla u_m\|_{2,Q_\tau} ^2+
\int_0^{\tau} \|u_m\|_{2,\Omega}^2\left({
[S_{ns/(s+n)}|\Omega |^{1/s+1/n-1/2}]^{2\lambda}\over a_\#}
\|{\bf E}\|_{q,\Omega}^2+\right.\\ \left. +
{1-\lambda\over 2}\left({2(\lambda +1)\over a_\#}\right)^{\lambda +1\over 1-
\lambda}[S_{ns/(s+n)}|\Omega |^{1/s+1/n-1/2}]^{2\lambda/(1-\lambda)}
\|{\bf E}\|_{q,\Omega}^{2/(1-\lambda)}
\right)\mathrm{dt}.
\end{align*}

Introducing the above inequality in (\ref{stepm}) we find
\begin{align*}
\|u_m\|_{2,\Omega}^2(\tau)
+a_\#\|\nabla u_m\|_{2,Q_\tau}^2
+b_\#\|u_m\|_{\ell ,\Sigma_\tau} ^\ell\leq
\|u_0\|_{2,\Omega}^2+
{2\over \ell\,'b_\#^{1/(\ell-1)}}\|h \|_{\ell\,',\Sigma_\tau} ^{\ell\,'}+
\\
+\int_0^{\tau} \|u_m\|_{2,\Omega}^2\left(
{2\over a_\#}S_{2n(n+\theta q)/(n^2+(n+2)\theta q)}^{2(n/q+\theta)}
|\Omega |^{2\theta/n}\|{\bf E}\|_{q,\Omega}^2+\right. \\ \left.
+[(4/a_\#)^{1+n/q+\theta} 
S_{2n(n+\theta q)\over n^2+(n+2)\theta q}^{2(n/q+\theta)}
|\Omega |^{2\theta/n}]^{q(q-n-\theta q)^{-1}}
\|{\bf E}\|_{q,\Omega}^{2(1-n/q-\theta )^{-1}}\right)
\mathrm{dt}.
\end{align*}
By applying the Gronwall inequality, we conclude (\ref{cotauinf})
for $u_m$, and consequently (\ref{cotauu}).

\subsection{Passage to the limit in (\ref{pbum}) as $m\rightarrow\infty$}

According to Section \ref{est}
 we may extract a subsequence of $\{u_m\}$ still denoted by $\{u_m\}$
such that $u_m\rightharpoonup u$ in
$L^2(0,T;V_{2,\ell})$,
and  $u_m\rightharpoonup u$ *-weakly in
$L^\infty(0,T;L^2(\Omega))$. In particular,  $u_m(T)\rightharpoonup z$ in
$L^2(\Omega)$, and using (\ref{bmm}) there exists a positive constant
such that
\[
\|b(u_m)u_m\|_{\ell/(\ell -1),\Sigma_T}\leq
\|u_m\|_{\ell,\Sigma_T}^{\ell -1}\leq C.
\]
 Thus, at least a subsequence $b(u_m)u_m$ weakly converges to $w$ in 
 $L^{\ell/(\ell-1)}(\Sigma_T)$.
 
Let us pass to the limit in (\ref{pbum}) by (\ref{perp})
(see  \cite[Theorem 4.1, p. 120]{show}).
Indeed,  by (\ref{perp}) we have the integration per parts formula
for all $\psi\in C^\infty([0,T])$ and $v\in V_{2,\ell}$,
\begin{align*}
(u_m(T),\psi(T)v)+
(u_0,\psi(0)v)=\int_0^T
\langle \partial_t u_m(t),\psi(t)v\rangle+
\langle \psi'(t)v,u_m(t)\rangle\mathrm{dt}=\\
=-\int_{Q_T}\psi(t)
( \mathsf{A}\nabla u_m(t)+T_m(u_m(t)){\bf E}) \cdot \nabla v\mathrm{dx}
\mathrm{dt}
+\\+\int_{\Sigma_T} \psi(t)(h(t)-b(u_m(t))u_m(t)) v\mathrm{ds}\mathrm{dt}+
\int_0^T \psi'(t)(v ,u_m(t))\mathrm{dt} ,
\end{align*}
where $(\cdot,\cdot)$ stands for the inner product of $L^2(\Omega)$.

Passing to the limit as $m\rightarrow\infty$ in the above equality,
we see that the triple $(z,u,w)$ satisfies 
\begin{align*}
(z,\psi(T)v)+(u_0,\psi(0)v)
=-\int_{Q_T}\psi(t)
( \mathsf{A}\nabla u(t)+u(t){\bf E}) \cdot \nabla v\mathrm{dx}
\mathrm{dt}
+\\+\int_{\Sigma_T} \psi(t)(h(t)-w(t)) v\mathrm{ds}\mathrm{dt}+
\int_0^T \psi'(t)(v ,u(t))\mathrm{dt} .
\end{align*}
If $\psi(T)=\psi(0)=0$, we find
\begin{align}\label{dtlim}
 \langle \partial_t u,v\rangle
=-\int_{\Omega}
( \mathsf{A}\nabla u+u{\bf E}) \cdot \nabla v\mathrm{dx}
+\int_{\Gamma}(h-w) v\mathrm{ds},
\quad\mbox{a.e. in }]0,T[.
\end{align}
If $\psi(T)=\psi(0)=0$, $u(T)=z$. 

It remains to prove that $w=b(u)u$.
Observe that 
the weak convergences are not sufficient
to that, since the argument of
the Minty trick fails,
although the coercivity (\ref{amin})
of $\mathsf{A}$ and the monotonicity property (\ref{bmm}) of $b$, 
because the existence of the term $T_m(u_m){\bf E} \cdot\nabla v$.
In order to apply the Aubin-Lions Lemma,
let us estimate $\partial_t u_m$ in $ L^{1}(0,T;(V_{2,\ell})')$.

For every $v\in V_{2,\ell}$, and for almost all $t\in ]0,T[$, we have
\begin{align*}
|\langle\partial_t u_m(t),v\rangle|\leq
a^\#\|\nabla u_m(t)\|_{2,\Omega}\|\nabla v\|_{2,\Omega}+\\ +
(b^\#\|u_m(t)\|_{\ell,\Gamma}^{\ell-1}+\|h(t)\|_{\ell\,',\Gamma})
\|v\|_{\ell,\Gamma}
+\|u_m(t)\|_{2q/(q-2),\Omega}\|{\bf E}(t)
\|_{q,\Omega}\|\nabla v\|_{2,\Omega}.
\end{align*}
Using (\ref{hold})-(\ref{sn2}), it follows that 
\begin{align*}
\|\partial_t u_m(t)\|_{(V_{2,\ell})'}\leq
a^\#\|\nabla u_m(t)\|_{2,\Omega}+
(b^\#\|u_m(t)\|_{\ell,\Gamma}^{\ell-1}+\|h(t)\|_{\ell\,',\Gamma})+\\
+S_{(2+\theta q)/(1+\theta q)}^{n/q+\theta}
|\Omega |^{\theta/n}\left(\|\nabla u_m(t)\|_{2,\Omega}^\lambda
\|u_m(t)\|_{2,\Omega}^{1-\lambda}+\|u_m(t)\|_{2,\Omega}\right)
\|{\bf E}(t)\|_{q,\Omega},
\end{align*}
where  $\lambda=n/q+\theta$ with $\theta=0$ if $n>2$ and
any $0<\theta<1-2/q$ if $n=2$.
 Since the inclusion of the spaces $L^r(0,T)\subseteq L^2(0,T)\subseteq
L^{\ell/(\ell-1)}(0,T) \subseteq L^1(0,T)$
 holds, applying the Minkowski and Young inequalities
we deduce
\begin{align*}
\|\partial_t u_m\|_{L^1(0,T;(V_{2,\ell})')}\leq
T^{1/ 2}
(a^\#+S_{(2+\theta q)/( 1+\theta q)})\|\nabla u_m\|_{2,Q_T}
+\\ +T^{1/ \ell}
b^\#(\|u_m\|_{\ell,\Sigma_T}^{\ell-1} +\|h\|_{\ell\,',\Sigma_T})+\\
+T^{1-{1/ r}}\left(|\Omega |^{\theta\over n(1-\lambda)}
T^{-\lambda}
+S_{(2+\theta q)/(1+\theta q)}^{n/q+\theta}
|\Omega |^{\theta/ n}\right)\|{\bf E}\|_{q,r,Q_T}
\|u_m\|_{2,\infty,Q_T}.
\end{align*}
The sequence on the right-hand side of this last relation is uniformly bounded
due to the estimates (\ref{cotauinf})-(\ref{cotauu}).
By the Aubin-Lions Lemma, 
 $\{u_m\}$ is relatively compact into $L^{q,{\ell/(\ell-1)}}(Q_T)$ for any $q<2n/(n-2)$,
and $L^{q,{\ell/(\ell-1)}}(\Sigma_T)$ for any $q<2(n-1)/(n-2)$.

Passing to the limit as $m\rightarrow\infty$ in (\ref{pbum}),
we see that the function $u$ satisfies (\ref{pbu}). 

\section{Minimum and maximum principles}
\label{sminmax}

 The objective of this section is
the proof of Theorem \ref{teomm}
by making recourse of the minimum and maximum principles.
It will be consequence of
 Propositions \ref{pmin} and \ref{pmax}.

\begin{proposition}[Minimum principle]\label{pmin}
Let $u$ solve (\ref{pbu}).
 Under $b_\#\geq 0$ in (\ref{bmm}),
$h\geq 0$ on $\Sigma_T$, and $u_0\geq 0$ in
$\Omega$, we have that $u\geq 0$ in $Q_T$ as well as its trace 
on $\Sigma_T$.
\end{proposition}
\begin{proof} The classical choice of  $v=u^-=\min\{u,0\}$ as a test function
in (\ref{pbum}) implies that, for almost all values $\tau$ in $]0,T[$
\begin{equation}
\int_\Omega (u^-)^{2}(\tau)\mathrm{dx}+{a_\#\over 2}\int_{Q_\tau}
 |\nabla u^-|^2\mathrm{dx}
\mathrm{dt}
\leq {1\over 2a_\#}\int_{Q_\tau} (u^-)^{2}|{\bf E}|^2\mathrm{dx}
\mathrm{dt},
\label{min}
\end{equation}
by considering the assumptions (\ref{amin}) and (\ref{bmm}) with $b_\#\geq 0$,
and the Young inequality.
Therefore, by applying the Gronwall inequality, we conclude that
$u^-=0$ in $Q_T$.

Introducing this fact in (\ref{min}), and letting
$\tau\rightarrow T$, it follows that the trace function $u\geq 0$
on $\Sigma_T$.
\end{proof}

In order to state our maximum principle,
we begin by establishing some preliminary results. The first one deals with
the well known interpolation result, which is a direct consequence of the
H\"older inequality.
\begin{lemma}\label{pp1}
If $w\in L^{q,q_1}(Q_T)\cap L^{r,r_1}(Q_T)$, then 
$w\in  L^{p,p_1}(Q_T)$, where
\[
{1\over p}={\lambda\over q}+{1-\lambda\over r},\qquad
{1\over p_1}={\lambda\over q_1}+{1-\lambda\over r_1},\qquad (\lambda\geq 1).
\]
Moreover,
\[
\|w\|_{p,p_1,Q_T}\leq 
\|w\|_{q,q_1,Q_T}^\lambda
\|w\|_{r,r_1,Q_T}^{1-\lambda}.
\]
\end{lemma}

We improve in Lemma \ref{sigma} (see also Remark \ref{rar})
a result established in \cite{aronson}.
\begin{lemma}\label{sigma}
If $w\in L^{2,\infty}(Q_T)\cap L^2(0,T;H^1(\Omega))$, then 
 $w\in L^{\sigma 2q/(q-2),\sigma 2r/(r-2)}(Q_T)$ for all 
 $q,r>2$ and $1-2/q<\sigma\leq 1+2(1-n/q-2/r)/n$ that satisfy
\[
\sigma\leq n(q-2)/[q(n-2)]\quad (q>2),
\qquad (
\sigma>2/q-2/r \mbox{ if }n=2).
\]
Moreover,
\[
\|w\|_{\sigma{2q\over q-2},\sigma{2r\over r-2},Q_T}^2\leq T^{\nu(\sigma)}
\left(\|w\|_{2,\infty,Q_T}^2+C_n(\sigma)\left(\|\nabla w\|_{2,Q_T}^2+
\|w\|_{2,\Sigma_T}^2\right)\right),
\]
where 
\begin{align*}
\mbox{if }n>2 \quad &
\left\{\begin{array}{l}
\nu(\sigma)=(1-n/q-2/r)/\sigma +n(1/\sigma-1)/2\\
C_n(\sigma)=2(S_{2,2})^2
\end{array}\right.
\\
\mbox{if }n=2 \quad &
\left\{\begin{array}{l}
\nu(\sigma)=(1-1/q-1/r)/\sigma -1/2\\
C_2(\sigma)=2\left(S_{2s/(s+2),2s/(s+2)}\right)^2\left(|\Omega |^{1/s}
+|\Gamma|^{1/s}\right)
\end{array}\right. 
\end{align*}
with $s=\nu^{-1}(2\sigma^{-1}(1/q-1/r)+1)$.
\end{lemma}
\begin{proof}
For any $\sigma>1-2/q$ and $s>\sigma 2q/(q-2)$ such that 
$\sigma 2r/(r-2)\leq 2/\lambda$,
applying successively the H\"older inequality, Lemma \ref{pp1},
and the Young inequality, we find
\begin{align*}
\|w\|_{\sigma {2q\over q-2},\sigma{2r\over r-2},Q_T}^2\leq T^\nu
\|w\|_{\sigma{2q\over q-2},2/\lambda,Q_T}^2\leq  T^\nu
\|w\|_{s,2,Q_T}^{2\lambda} 
\|w\|_{2,\infty,Q_T}^{2(1-\lambda)}\leq\\
\leq T^\nu\left(
\lambda  \|w\|_{s,2,Q_T}^2+(1-\lambda )
\|w\|_{2,\infty,Q_T}^{2}
\right),
\end{align*}
with 
\[
\lambda ={2s\over s-2}\left[{1\over 2}-{1\over\sigma}
\left({1\over 2}-{1\over q}\right)\right]
\quad\mbox{and}\quad \nu={1\over \sigma}\left(1-{2\over r}\right)-\lambda .
\]

If $n>2$, we choose $s=2^*=2n/(n-2)$ then the Sobolev embedding can  be applied
concluding the desired result.

If $n=2$, we choose $s=2+4\nu^{-1}\left[{1\over 2}-{1\over\sigma}
\left({1\over 2}-{1\over q}\right)\right]>2$ which implies that
$\lambda=\nu+2\left[{1\over 2}-{1\over\sigma}
\left({1\over 2}-{1\over q}\right)\right]$.
Therefore
we use the Sobolev embedding $W^{1,2s/(s+2)}(\Omega) \hookrightarrow L^s(\Omega)$
followed by the H\"older inequality in order to determine the constant $C_2
(\sigma)$,
namely,
\[
\|w\|_{s,\Omega}^2\leq (S_{2s/(s+2),2s/(s+2)})^2\left(|\Omega |^{1/s}
\|\nabla w\|_{2,\Omega}+|\Gamma|^{1/s}\|w\|_{2,\Gamma}\right)^2,
\]
finishing then the proof of Lemma \ref{sigma}.
\end{proof}

\begin{remark}\label{rar}
The existence of $\sigma$ satisfying $1-2/q<\sigma\leq 1+2(1-n/q-2/r)/n$
is given by $r>2$, while
 $\min\{ n(q-2)/[q(n-2)], 1+2(1-n/q-2/r)/n\}= 1+2(1-n/q-2/r)/n$ if and only if
 \[
 {n\over q}+{2-n\over r}\leq 1.\]
 For $n=2$, the existence of $\sigma$ satisfying
$\max\{2/q-2/r,1-2/q\}\leq 1<\sigma\leq 2(1-1/q-1/r)$ is guaranteed by $q,r>2$.
\end{remark}

Set
\begin{equation}\label{ninf}
 \|w\|_{\infty,Q_T}=
\lim_{N\rightarrow\infty} \|w\|_{p\chi^{N},q\chi^{N},Q_T},\quad\forall
w\in \cap_{1\leq p,q<\infty} L^{p,q}(Q_T),
 \end{equation}
where $p,q\geq 1$, and $\chi> 1$.

Next we improve the technical result, which  involves 
an additional term.
 \begin{lemma}\label{techni}
 Let $p,q> 1$, and $\chi<1$.
 If $w\in \cap_{1\leq p,q<\infty} L^{p,q}(Q_T)$ verifies
  \begin{equation}
 \|w\|_{p(1/\chi)^{m+1},q(1/\chi)^{m+1},Q_T}\leq (P\chi^{-2m})^{\chi^{m}}
 \|w\|_{p(1/\chi)^{m},q(1/\chi)^m,Q_T}+P_1P_2^{\chi^m},
 \end{equation}
for some  constants $P\geq 1$, $P_1\geq 0$, and $0<P_2\leq P$,
and for any $m\in\mathbb{N}_0$, then
 \begin{equation}\label{tech}
{\rm ess} \sup_{Q_T}|w|\leq P^{1\over 1-\chi}\chi^{-{\chi\over(1-\chi)^2}}
 \|w\|_{p,q,Q_T}+P_1
\sum_{i\geq 0}P^{\chi^{i+1}\over 1-\chi}
\chi^{{i\chi^{i+2}-( i+1)\chi^{i+1}\over (1-\chi)^2}}
P_2^{\chi^i}.
 \end{equation}
\end{lemma}
\begin{proof}
By induction, we have for all $N\in\mathbb{N}$
\[ 
 \|w\|_{p\chi^{-N},q\chi^{-N},Q_T}\leq P^{a_{0,N}}\chi^{-b_{0,N}}
 \|w\|_{p,q,Q_T}+P_1\sum_{i=0}^{N-1}P^{a_{i+1,N}}\chi^{-b_{i+1,N}}
P_2^{\chi^i},
 \]
 where
 \begin{align*}
a_{j_0,N}&=\sum_{j=j_0}^{N-1}\chi^{j}={\chi^{j_0}-\chi^{N}\over 1-\chi}
;\\
 b_{j_0,N}&=\sum_{j=j_0}^{N-1}j\chi^{j}={j_0\chi^{j_0}+(1-j_0)
\chi^{j_0+1}-N \chi^N+(N-1)\chi^{N+1}
\over (1-\chi)^2}.
\end{align*}
Using  d'Alembert's ratio criterium,  the second series in (\ref{tech}) 
is convergent if $
P_2^{\chi}\chi^{(i+1)\chi^{i+1}}<P^{\chi^{i+1}}$. Indeed, this inequality is
true for all $i\in\mathbb{N}_0$,
for $\chi\leq 1$, $0<P_2\leq P$, and $P\geq 1$.

Letting  $N\rightarrow \infty$, we find (\ref{tech})
by the definition (\ref{ninf}).
\end{proof}

Finally, we are in position to establish the upper bound of any solution
of (\ref{pbum}), if $2/r+n/q<1$.
\begin{proposition}[Maximum principle]\label{pmax}
Let $u$ solve (\ref{pbu}). Under $b_\#>0$ in (\ref{bmm}), and
 $2/r+n/q<1$, we have
\begin{align}
u&\leq
P^{\sigma\over \sigma-1}\sigma^{\sigma\over(\sigma-1)^2}
 T^{\nu(1)\over 2}
\left(\|u\|_{2,\infty,Q_T}^2+C_n(1)\left(\|\nabla u\|_{2,Q_T}^2+
\|u\|_{2,\Sigma_T}^2\right)\right)^{1/2}\nonumber\\
&+P_1
\sum_{i\geq 0}P^{\sigma^{-i}\over \sigma-1}
\sigma^{{( i+1)\sigma^{-i+1}-i\sigma^{-i}\over (\sigma-1)^2}}
P_2^{\sigma^{-i}}:=\mathcal{M}\quad \mbox{in }Q_T;\label{cotamax1} \\
\label{cotamax2}
u&\leq {1\over b_\#}\left(\mathcal{M}+P_1\right) \quad \mbox{
on }\Sigma_T,
\end{align}
if provided by the smallness condition $P_2\leq P$, with
\begin{align*}
P&=\left({2T^{\nu(\sigma)}
\max\{1,C_n(\sigma)\}\over a_\#\min\{1,a_\#,b_\#\}}\right)^{1/2}
\|{\bf E}\|_{q,r,Q_T}\geq 1
;\\
P_1&=\left(
\max\{1, \|  u_0\|_{\infty,\Omega},\|  h\|_{\infty,\Sigma_T}\}\right)^{1/2};\\
P_2&=\left({T^{\nu(\sigma)}\max\{1,C_n(\sigma)\}\over\min\{1,a_\#,b_\#\}} 
\left(|\Omega |+( b_\#(\ell-2)+
\max\{1,1/b_\#\})|\Sigma_T|\right)\right)^{1/2}
;\\
\sigma &=1+\frac2n\left(1-{2\over r}-{n\over q}\right),
\end{align*}
where $\nu$ and $C_n$ are introduced in Lemma \ref{techni}.
\end{proposition}
\begin{proof}
 Set $\theta=1-2/r-n/q>0$.
Arguing as in \cite{aronson},
 the first step involves showing that, for almost all values $\tau$ in $]0,T[$
\begin{align}
 {1\over\beta+1}
\int_\Omega (u^+)^{\beta+1}(\tau)\mathrm{dx}+{a_\#\beta \over 2}\int_{Q_\tau}
(u^+)^{\beta-1} |\nabla u^+|^2\mathrm{dx}
\mathrm{dt}+\nonumber \\+
{b_\#\over\beta+1}\int_{\Sigma_\tau} (u^+)^{\beta+1}
\mathrm{ds}\mathrm{dt}
\leq  {\ell-2\over \beta+\ell-1}b_\#|\Sigma_\tau|+
 {1\over\beta+1}
\int_\Omega (u^+)^{\beta+1}(0)\mathrm{dx}+\nonumber\\+
{\beta\over 2a_\#}\int_{Q_\tau} (u^+)^{\beta+1}|{\bf E}|^2\mathrm{dx}
\mathrm{dt}+
{|\Sigma_\tau|\over (\beta+1)(b_\#)^\beta}\|  h\|_{\infty,\Sigma_\tau}
^{\beta+1}, \label{step1}
\end{align}
where $\beta\geq 1$, and $u^+=\max\{u,0\}$.
 Let us take  $v=\chi(t,\tau)\mathcal{G}(u)$
 as a test function in (\ref{pbum}), where
 $\chi(t,\tau)$ is the characteristic function of the open interval $]0,\tau[$,
with $\tau$ being a fixed number lesser than $T$, and
\[
\mathcal{G}(u)=
\left\{\begin{array}{ll}
(u^+)^\beta& \mbox{for } -\infty< u\leq M\\
M^{\beta-1}u& \mbox{for } M\leq u<+\infty 
\end{array}\right. .
\] 
Applying (\ref{amin}), it follows that
\begin{align*}
\int_{Q_\tau}\partial_t [\mathcal{H}(u)]\mathrm{dx}\mathrm{dt}
+{a_\#\over 2}\int_{Q_\tau}
\mathcal{G}'( u)|\nabla u^+|^2\mathrm{dx}
\mathrm{dt}+\int_{\Sigma_\tau} b(u)u\mathcal{G}(u)\mathrm{ds}\mathrm{dt}\leq\\
\leq {1\over 2a_\#}
\int_{Q_\tau} |u^+{\bf E}|^2\mathcal{G}'(u)\mathrm{dx}
\mathrm{dt}+\int_{\Sigma_\tau} h (u^+)^\beta\mathrm{ds}\mathrm{dt},
\end{align*}
with $ \mathcal{H}'(u)= \mathcal{G}(u)$, and considering Remark \ref{dtc}.
As the last boundary integral in the above inequality is new,
we analyze it  separately.
 Applying the H\"older and Young inequalities,
we deduce
\[ 
\int_{\Sigma_\tau} h (u^+)^\beta\mathrm{ds}\mathrm{dt}\leq 
{|\Sigma_\tau|\over (\beta+1)(b_\#)^\beta}\|  h\|_{\infty,\Sigma_\tau}
^{\beta+1}+
{b_\#\beta\over \beta+1}\int_{\Sigma_\tau}
  (u^+)^{\beta+1}\mathrm{ds}\mathrm{dt}.
\]
Letting the parameter $M$ tend to infinity
(since $\mathcal{G}'(u)\leq \beta(u^+)^{\beta-1}$), 
and applying (\ref{bmm}) with $b_\#>0$, we
 compute the boundary integral on the left hand-side as follows
\[
\int_{\Sigma_\tau}
  (u^+)^{\beta+1}\mathrm{ds}\mathrm{dt}\leq {\ell-2\over \beta+\ell-1}
|\Sigma_\tau|+
\int_{\Sigma_\tau}
  (u^+)^{\beta+\ell-1}\mathrm{ds}\mathrm{dt},
\]
 finding (\ref{step1}).

The second step involves showing that (\ref{step1}) implies
\begin{align}
\|w^\sigma \|_{{2q\over q-2},{2r\over r-2},Q_T}^{2/\sigma}\leq 
{T^\nu\max\{1,C_n\}\over
\min\{1,a_\#,b_\#\}}\left(
|\Omega| \|  u_0\|_{\infty,\Omega}^{\beta+1}+
\right. \nonumber \\
\left.+ {(\beta +1)^2\over 2a_\#}
\|{\bf E}\|_{q,r,Q_T}^2
\|w\|_{{2q\over q-2},{2r\over r-2},Q_T}^2+
 (\ell-2)b_\#|\Sigma_T|+
{|\Sigma_T|\over (b_\#)^\beta}\|  h\|_{\infty,\Sigma_T}
^{\beta+1}\right), \label{cotaw}
\end{align}
with $w=(u^+)^{(\beta+1)/2}$.

Multiplying (\ref{step1}) by $\beta+1$ we have
\begin{align}
\|w\|_{2,\infty,Q_T}^2
+a_\#\|\nabla w\|_{2,Q_T}^2
+b_\#\| w\|_{2,\Sigma_T}^2\leq 
|\Omega| \|  u_0\|_{\infty,\Omega}^{\beta+1}+\nonumber \\
+ {(\beta +1)^2\over 2a_\#}
\|{\bf E}\|_{q,r,Q_T}^2\|w\|_{{2q\over q-2},{2r\over r-2},Q_T}^2
+ (\ell-2)b_\#|\Sigma_T|+
{|\Sigma_T|\over (b_\#)^\beta}\|  h\|_{\infty,\Sigma_T}^{\beta+1}.\label{usigm}
\end{align}
On other hand, by taking $\sigma=1+2\theta /n$ $(n\geq 2)$, that is
$\nu=0$,
Lemma \ref{sigma} guarantees that (\ref{cotaw}) holds.

Next, returning to $u^+$, (\ref{cotaw}) becomes
\[
\varphi_{N+1}=\|u^+\|_{\sigma^{N+1}{2q\over q-2},\sigma^{N+1}{2r\over r-2},
Q_T}\leq (P\sigma^{N})^{\sigma^{-N}}\varphi_N+P_1P_2^{\sigma^{-N}},
\]
with $N\in\mathbb{N}$ and $(\beta+1)/2=\sigma^N$
 stand for the iterative argument (cf. Lemma \ref{techni}).
Therefore, we conclude (\ref{cotamax1}) making recourse of Lemma
\ref{sigma} with $\sigma=1$.

Finally, introducing the  upper bound $\mathcal{M}$ in (\ref{usigm}) we find
\begin{align*}
b_\#\| u^+\|_{\sigma^N,\Sigma_T}
\leq \sigma^{N\sigma^{-N}}\left({|\Sigma_T|\over 2a_\#}
\|{\bf E}\|_{q,r,Q_T}^2\right)^{\sigma^{-N}}\mathcal{M} 
+P_1P_2^{\sigma^{-N}}.
\end{align*}
Applying directly the definition (\ref{ninf})
we conclude (\ref{cotamax2}).
\end{proof}

\section{Proof of Theorem \ref{teomf}}
\label{steomf}

Let us reformulate Lemma \ref{sigma} under exponents $p$ being lesser or
equal than $n/(n-1)$.
\begin{lemma}\label{vpp1}
If $w\in L^{1,\infty}(Q_T)\cap L^{p_1}(0,T;W^{1,p}(\Omega))$, then 
 $w\in L^{\bar p,\bar q}(Q_T)$ for all $1\leq p\leq n/(n-1)$,
 $1\leq p_1<\bar q$ and 
\begin{equation}\label{pqbar}
{1\over \bar p}+{p_1\over \bar q} \left(1+{1\over n}-{1\over p}\right)=1.
\end{equation}
For $\lambda=p_1/\bar q<1$ we have
\begin{align}
\|w\|_{\bar p,\bar q,Q_T}\leq
\left(S _p (1+|\Omega|^{1/n}S_1)\right)^\lambda
\|\nabla w\|_{p,p_1,Q_T}^\lambda \|w\|_{1,\infty,Q_T}^{1-\lambda}+
\nonumber \\ +
(S_pS_1)^\lambda T^{1/\bar q}|\Omega|^{\lambda (1/p+1/n-1)}
\|w\|_{1,\infty,Q_T}.\label{wp1}
\end{align}
\end{lemma}
\begin{proof}
Let us begin by establishing the following
 correlation between the Sobolev constants, 
 for all $1\leq p\leq n/(n-1)$,
\begin{align*}
\|w\|_{pn/(n-p),\Omega}\leq S_p\left(
\|\nabla w\|_{p,\Omega}+|\Omega|^{1/p-(n-1)/n}
\|w\|_{n/(n-1),\Omega}\right)\leq \\
\leq S _p 
\left((1+|\Omega|^{1/n}S_1)
\|\nabla w\|_{p,\Omega}+|\Omega|^{1/p-(n-1)/n}
S_1\|w\|_{1,\Omega}\right).
\end{align*}
Thanks to Lemma \ref{pp1} with the above inequality 
we conclude (\ref{wp1}).
\end{proof}

As in the proof of Theorem \ref{teom},
let us first take the existence of approximate solutions 
 $u_m$ in 
$ L^{2,\infty}(Q_T)\cap L^2(0,T;V_{2,\ell})$ such that
 $\partial_t u_m\in L^2(0,T;(V_{2,\ell})')$,
for each $m\in\mathbb{N}$, of the variational problem
\begin{align}
\int_0^T\langle\partial_t u_m,v\rangle\mathrm{dt} +\int_{Q_T}
( \mathsf{A}\nabla u_m+T_m(u_m){\bf E}) \cdot \nabla v\mathrm{dx}
\mathrm{dt}+\nonumber \\
+\int_{\Sigma_T} b(u_m)u_mv\mathrm{ds}\mathrm{dt}
=\int_{Q_T} {mf\over m+|f|} v\mathrm{dx}\mathrm{dt}+
\int_{\Sigma_T} h v\mathrm{ds}\mathrm{dt},\label{pbumf}
\end{align}
for all $v\in L^2(0,T;V_{2,\ell})$.

 Next, we deal to the
derivation of uniform estimates, and the passage to the limit in (\ref{pbumf}).

\subsection{$L^{1,\infty}(Q_T)$- and $L^{\ell-1}(\Sigma_T)$-estimates
 (\ref{cotauinfp}) for  $u_m$}
\label{sumf1}

Let $\varepsilon\in ]0,m[$ be arbitrary.
Choosing $v=\chi(t,\tau)T_1(u_m/\varepsilon)$ as a test function in (\ref{pbum}),
 where 
 $\chi(t,\tau)$ is the characteristic function of the open interval $]0,\tau[$,
with $\tau$ being a fixed number lesser than $T$,
and $T_1$ is the truncating function (\ref{trt})
with $m=1$, we obtain
\begin{align*}
\int_0^\tau
{d\over dt}\int_\Omega\left[\int_0^{u_m}T_1(z/\varepsilon)\mathrm{dz}\right]
\mathrm{dx}\mathrm{dt}+
b_\# \int_{\Sigma_\tau[|u_m|>\varepsilon]}
| u_m|^{\ell-1}\mathrm{ds}\mathrm{dt} \leq \\
\leq \int_{Q_\tau[|u_m|<\varepsilon]}
|{\bf E}||\nabla u_m|\mathrm{dx}\mathrm{dt} 
+\int_{Q_\tau}|f|\mathrm{dx}\mathrm{dt}
+\int_{\Sigma_\tau}|h|\mathrm{ds}\mathrm{dt},
\end{align*}
taking (\ref{amin}) and (\ref{bmm}) into account.
Passing to the limit as $\varepsilon$ tends to zero, (\ref{cotauinfp}) holds.

\subsection{$L^{p}$-estimate (\ref{cotauup}) to the gradient of $u_m$}
\label{sumf2}

Thanks to the estimate \cite[Lemma 4.4.5]{lap}
\begin{align*}
\|\nabla u_m\|_{p,Q_T}^p\leq \left(\int_{Q_T}{ |\nabla u_m|^2\over
(1+| u_m|)^{\delta +1}}\mathrm{dx}\mathrm{dt}\right)^{p/2}
\left(|Q_T|^{n\over p(n+1)}
+\right. \\ \left.
 +\|u_m\|_{p(n+1)/n,Q_T} \right)^{p(n+1)(2-p)\over 2n},
\end{align*}
for any $1<p<(n+2)/(n+1)$, and $\delta=(2-p)(n+1)/n-1\in ]0,1[$, considering that
by Lemma \ref{vpp1} with $\bar p=\bar q=p(n+1)/n$
and $p_1=p<(n+2)/(n+1)$, implies
\begin{equation}\label{pQT}
\|u_m\|_{{p(n+1)\over n},Q_T}\leq (Z_{1})^{n/(n+1)}
\|\nabla u_m\|_{p,Q_T}^{n\over n+1}
\mathcal{Z}^{1/(n+1)}+(Z_2)^{n/(n+1)} \mathcal{Z},
\end{equation}
 we deduce
\begin{align*}
{p\over 2}\|\nabla u_m\|_{p,Q_T}^p\leq 
{p\over 2} (Z_1)^{2-p}
\left(\int_{Q_T}{ |\nabla u_m|^2\over
(1+| u_m|)^{\delta +1}}\mathrm{dx}\mathrm{dt}\right)\mathcal{Z}^{2-p\over n}+\\
+\left(\int_{Q_T}{ |\nabla u_m|^2\over
(1+| u_m|)^{\delta +1}}\mathrm{dx}\mathrm{dt}\right)^{p/2}
\left(|Q_T|^{n\over p(n+1)}
 +(Z_2)^{n/(n+1)} \mathcal{Z}\right)^{p(n+1)(2-p)/(2n)},
\end{align*}
observing that the term in LHS is rearranged by using
$(a+b)^\varkappa\leq a^\varkappa +b^\varkappa$ with
$\varkappa= p(n+1)(2-p)/(2n)<1$, and
the Young inequality $ab\leq pa^{2/p}/2+(1-p/2)b^{2/(2-p)}$.
Reusing the Young inequality, and using
$(a+b)^\varkappa\leq 2(a^\varkappa +b^\varkappa)$ with
$\varkappa= p(n+1)/n<2$, we rewrite the above inequality as
\begin{align}
\|\nabla u_m\|_{p,Q_T}^p\leq 
\left(\int_{Q_T}{ |\nabla u_m|^2\over
(1+| u_m|)^{\delta +1}}\mathrm{dx}\mathrm{dt}\right)
\left( (Z_1)^{2-p}\mathcal{Z}^{2-p\over n}+1\right)+\nonumber \\
+{2(2-p)\over p}
\left(|Q_T| +(Z_2)^{p}\mathcal{Z}^{p(n+1)/n}\right).\label{numpp}
\end{align}

Thus, 
it remains to estimate the integral term.
From $L^1$-data theory (see, for instance, 
\cite{bopo,lap} and the references therein),
let us choose 
\[ 
v=-{{\rm sign}(u_m)}(1+|u_m|)^{-\delta} \in
L^2(0,T;W^{1,2}(\Omega))\cap L^\infty(Q_T),\quad
\mbox{for }\delta  >0,
\]
as a test function in (\ref{pbum}).
Using (\ref{amin}) and (\ref{bmm}), it follows that
\begin{align}
a_\#\int_{Q_T}{\delta |\nabla u_m|^2\over
(1+| u_m|)^{\delta +1}}\mathrm{dx}\mathrm{dt}
\leq  {1\over 1-\delta}(|Q_T|+ \|u_m\|_{1,\infty,Q_T})+\nonumber \\
+\|f\|_{1,Q_T}+\|h\|_{1,\Sigma_T}
+b^\#\int_{\Sigma_T}|u_m|^{\ell-1}\mathrm{ds}\mathrm{dt}+\nonumber \\
 +{\delta\over 2a_\#}\|{u_m\over
(1+| u_m|)^{\delta+1\over 2}}\|_{2q/(q-2),2r/(r-2),Q_T}^2\|{\bf E}\|_{q,r,Q_T}^2+
{a_\#\delta\over 2}\|{\nabla u_m\over
(1+| u_m|)^{\delta+1\over 2}}\|_{2,Q_T}^2.\label{uell}
\end{align}

Since (\ref{pqrn}) implies that $\bar p=(1-\delta)q/(q-2)$, 
 $\bar q=(1-\delta)r/(r-2)$, and $p_1=p$ satisfy (\ref{pqbar}),
then we compute
\begin{align*}
\|{u_m\over
(1+| u_m|)^{\delta+1\over 2}}\|_{{2q\over q-2},{2r\over r-2},Q_T}^2\leq 
\|u_m\|_{\bar p,\bar q,Q_T}^{1-\delta}\leq\\
\leq
(Z_{1})^{p(r-2)/r}
\|\nabla u_m\|_{p,Q_T}^{p(r-2)/r}
\mathcal{Z}^{(1-\delta)(1-\lambda)}+(Z_2)^{p(r-2)/r} \mathcal{Z}^{1-\delta},
\end{align*}
with $\lambda=p(r-2)/[(1-\delta)r]<1$ because $r(2-p)<2np$.
Inserting these two above inequalities into (\ref{numpp}),
 we conclude
\begin{align*}
{2\over r}\|\nabla u_m\|_{p,Q_T}^p\leq {2(2-p)\over p}
\left(|Q_T| +(Z_2)^{p}\mathcal{Z}^{p(n+1)/n}\right)+\\
+{2\over a_\#\delta}
\left( (Z_1)^{2-p}\mathcal{Z}^{2-p\over n}+1\right)
\left( {|Q_T|\over 1-\delta}+\mathcal{Z}
((2-\delta)/(1-\delta)+b^\#/b_\#)\right)+\\
+{1\over (a_\#)^2}
\left( (Z_1)^{2-p}\mathcal{Z}^{2-p\over n}+1\right)\|{\bf E}\|_{q,r,Q_T}^2
(Z_2)^{p(r-2)/r} \mathcal{Z}^{1-\delta}+\\
+{2\over r(a_\#)^r}
\left( (Z_1)^{2-p}\mathcal{Z}^{2-p\over n}+1\right)^{r/2}\|{\bf E}\|_{q,r,Q_T}^r
(Z_1)^{p(r-2)/2} \mathcal{Z}^{r(1-\delta)-p(r-2)\over 2},
\end{align*}
and therefore replacing $\delta$ by its value (\ref{cotauup}) holds.

\subsection{Estimate of $\partial_t u_m$  in $L^{1}(0,T;( W^{1,p'}(\Omega))')$}
\label{sumf3}

For every $v\in W^{1,p'}(\Omega)\hookrightarrow
C(\bar\Omega)$, and for almost all $t\in ]0,T[$, we have
\begin{align*}
|\langle\partial_t u_m(t),v\rangle|\leq
a^\#\|\nabla u_m(t)\|_{p,\Omega}\|\nabla v\|_{p',\Omega}+\\ +
(b^\#\|u_m(t)\|_{\ell-1,\Gamma}+\|h\|_{1,\Gamma})\|v\|_{\infty,\Gamma}
+\|u_m(t)\|_{pq/(q-p),\Omega}\|{\bf E}(t)
\|_{q,\Omega}\|\nabla v\|_{p',\Omega}.
\end{align*}
Similarly to (\ref{hold})-(\ref{sn2}), we have
\begin{align*}
 \|u_m\|_{{pq\over q-p},\Omega}\leq  \|u_m\|_{{np\over n-p},\Omega}^\lambda
  \|u_m\|_{1,\Omega}^{1-\lambda},\quad 
\lambda={np\over np+p-n}\left(1-{1\over p}+{1\over q}\right);\\
\|u_m\|_{np/(n-p),\Omega}\leq S_{p,\ell-1}\left(
\|\nabla u_m\|_{p,\Omega}+\|u_m\|_{\ell-1,\Gamma}\right),
\end{align*}
then it follows that 
\begin{align*}
\|\partial_t u_m(t)\|_{( W^{1,p'}(\Omega))'}\leq
(a^\#+S_{p,\ell-1})\|\nabla u_m(t)\|_{p,\Omega}+\|h\|_{1,\Gamma}
+\\ +
(b^\#C_\infty +S_{p,\ell-1})\|u_m(t)\|_{\ell-1,\Gamma}
+\|u_m(t)\|_{1,\Omega}
\|{\bf E}(t)\|_{q,\Omega}^{1/(1-\lambda)},
\end{align*}
where $C_\infty$ denotes the constant of continuity of the Morrey embedding
$W^{1,p'}(\Omega)\hookrightarrow C(\bar\Omega)$.
 Since (\ref{pqrn}) means that $1/(1-\lambda)=q(p(n+1)-n)/[p(q-n)]=r/p$,
 we deduce
\begin{align*}
\|\partial_t u_m\|_{L^{1}(0,T;( W^{1,p'}(\Omega))')}\leq
(a^\#+S_{p,\ell-1}) T^{1-1/p}\|\nabla u_m\|_{p,Q_T}+
\|h\|_{1,\Sigma_T}
+\\
+(b^\#C_\infty +S_{p,\ell-1}) T^{1-1/(\ell-1)}\|u_m\|_{\ell-1,p,\Sigma_T}+
\|u_m\|_{1,\infty,Q_T} T^{1-1/p}
\|{\bf E}\|_{q,r,Q_T}^{r/p}.
\end{align*}
The sequence on the right-hand side of this last relation is uniformly bounded
due to the estimates (\ref{cotauinfp})-(\ref{cotauup}).

\subsection{Passage to the limit in (\ref{pbum}) as $m\rightarrow\infty$}

By Sections \ref{sumf1} and \ref{sumf2}
 we may extract a subsequence of $\{u_m\}$ still denoted by $\{u_m\}$
such that $u_m\rightharpoonup u$ in
$L^\iota (0,T;V_{p,\ell-1})$ for $\iota=\min\{p,\ell-1\}$.

Since $V_{p,\ell-1}\hookrightarrow\hookrightarrow L^q(\Omega)\hookrightarrow
 L^p(\Omega)\hookrightarrow W^{1,p'}(\Omega)$  for any $q<pn/(n-p)$,
and $V_{p,\ell-1}\hookrightarrow\hookrightarrow L^q(\Gamma)\hookrightarrow
 L^p(\Gamma)\hookrightarrow W^{1,p'}(\Omega)$  for any $q<p(n-1)/(n-p)$,
according to Section \ref{sumf3} the Aubin-Lions Lemma yields that 
 $\{u_m\}$ is relatively compact into $L^{q,\iota }(Q_T)$ for any $q<pn/(n-p)$,
and $L^{q,\iota }(\Sigma_T)$ for any $q<p(n-1)/(n-p)$.
In particular, $|u_m|^{\ell-2}$
 converges to $|u|^{\ell-2}$ a.e. on $\Sigma_T$.
As the Nemytskii  operator  $b$ is  continuous,
 $b(u_m)$ strongly converges to $b(u)$ in $L^{(\ell-1)/(\ell-2)}(\Sigma_T)$.
Therefore, (\ref{pbumf}) passes to the limit as $m$ tends to infinity,
concluding that $u$ solves (\ref{pbuf}).

\section{The case of $b_\#=0$}

The following proofs pursue the ones that are established in Sections
\ref{steom}, and \ref{steomf}.
Therefore, we only focus our attention to the
quantitative  estimates.

\subsection{Proof of Theorem \ref{tb0}}

We observe that (\ref{stepm}) reads
\begin{align*}
{1\over 2}\int_{\Omega}|u_m|^2(\tau)\mathrm{dx}
+a_\#\int_{Q_\tau}|\nabla u_m|^2\mathrm{dx}\mathrm{dt}
\leq 
{1\over 2}\|u_0\|_{2,\Omega}^2+\\ +
\int_0^{\tau} \|u_m\|_{2q/(q-2),\Omega}\|{\bf E}\|_{q,\Omega}
\|\nabla u_m\|_{2,\Omega}\mathrm{dt}
+\int_0^\tau\|h \|_{(2_*)',\Gamma} \|u_m\|_{2_*,\Gamma}\mathrm{dt} .
\end{align*}
The last term of RHS is computed as follows
\begin{align*}
\int_0^\tau\|h \|_{(2_*)',\Gamma} \|u_m\|_{2_*,\Gamma}\mathrm{dt}\leq
 K_{s}|\Omega|^{{1\over s}
-{1\over 2}}
\int_0^\tau \|h \|_{(2_*)',\Gamma}
\left(\|\nabla u_m\|_{2,\Omega}+\|u_m\|_{2,\Omega}\right)\mathrm{dt}
\\
\leq \left({1\over a_\#}+{1\over 2}\right)( K_{s})^2|\Omega|^{2/s-1}
\|h \|_{(2_*)',2,\Sigma_\tau}^2+
{a_\#\over 4}\|\nabla u_m\|_{Q_\tau}^2+
{1\over 2}\int_0^\tau\|u_m\|_{2,\Omega}^2\mathrm{dt},
\end{align*}
where $s=2_*n/(2_*+n-1)$. Thus, we may proceed as Section \ref{est}
to conclude (\ref{cotaub0}), and subsequently (\ref{cotaub1}).
The remaining proof follows {\it mutatis mutandis}.

\subsection{Proof of Theorem \ref{tb0f}}

The estimate (\ref{cotaub0p}) is a direct consequence of Section \ref{sumf1}.
To show that (\ref{cotauub0p}) holds,
it suffices to pay attention in (\ref{uell})
to the boundary integral, which obeys the following lemma.
\begin{lemma}\label{luell}
If $p(n-1)<n$ and $\ell\leq p+1$, then
\begin{align*}
\int_{\Sigma_T}|v|^{\ell-1}\mathrm{ds}\mathrm{dt}\leq
T^{1-{(\ell-1)/p}} |\Gamma |^{1-{(\ell-1)/ p_*}}
K_p^{\ell-1}\left(2\| \nabla v\|_{p,Q_T}+\right.\\ \left. +
(S_{1}^{n(p-1)\over n-p(n-1)}|\Omega|^{n(p-1)^2\over (n-p(n-1))p}
+S_{1}^{n(p-1)\over p})\|v\|_{1,p,Q_T}\right)^{\ell-1}.
\end{align*}
for every $v\in L^p(0,T;W^{1,p}(\Omega))\cap L^{1,p}(Q_T)$.
\end{lemma}
\begin{proof}
We apply firstly the H\"older inequality, and secondly the trace embedding
for $\ell-1\leq p_*$ to obtain
\[
\|v\|_{\ell-1,\Gamma}\leq |\Gamma |^{1/(\ell-1)-1/p_*}
K_p(\| \nabla v\|_{p,\Omega}+\|v\|_{p,\Omega}).
\]
We separately apply the interpolative inequality and after
the Sobolev embedding and the H\"older inequality, obtaining
\begin{align*}
\|v\|_{p,\Omega}\leq \|v\|_{n/(n-1),\Omega}^\lambda \|v\|_{1,\Omega}^{1-\lambda}
\qquad (\lambda=n(p-1)/p<1)\\
\leq S_{1}^\lambda |\Omega|^{\lambda(1-1/p)}
\|\nabla v\|_{p,\Omega}^\lambda\|v\|_{1,\Omega}^{1-\lambda}+S_1^\lambda
\| v\|_{1,\Omega}.
  \end{align*}
  Thus, inserting this last inequality into the above one, we deduce
\[
\|v\|_{\ell-1,\Gamma}\leq |\Gamma |^{{1\over\ell-1}-{1\over p_*}}
K_p\left(2\| \nabla v\|_{p,\Omega}+
(S_{1}^{\lambda/(1-\lambda)}|\Omega|^{\lambda(p-1)\over (1-\lambda)p}
+S_{1}^\lambda)\|v\|_{1,\Omega}\right).
\]
  Integrating in time, and applying the H\"older inequality,
 the proof is  complete.
\end{proof}

Then, Lemma \ref{luell} implies that (\ref{cotauup}) for $u_m$ is rewritten as 
\begin{align*}
\|\nabla u_m\|_{p,Q_T}^p \leq \mathcal{B}
+2^{2\ell-3}\beta\| \nabla u_m\|_{p,Q_T}^{\ell-1}
+ \\ +2^{\ell-2}\beta\left(
(S_{1}^{n(p-1)\over n-p(n-1)}|\Omega|^{n(p-1)^2\over (n-p(n-1))p}
+S_{1}^{n(p-1)\over p})T\mathcal{Z}\right)^{\ell-1}.
\end{align*}
If $\ell-1=p$, supposing that $\beta<2^{1-2p}$ then we find (\ref{cotauub0p}).
If $\ell-1<p$, we conclude (\ref{cotauub0p}) by considering the
Young inequality. The proof of Theorem \ref{tb0f} follows the argument
of Section \ref{steomf}.

\end{document}